\documentclass[smallcondensed]{svjour3}
\smartqed
\usepackage{graphicx,subfig}
\usepackage{stmaryrd}
\usepackage{commath}
\usepackage{hyperref}

\usepackage{cleveref}
\usepackage{url}
\usepackage{amsmath}
\usepackage{amssymb}
\usepackage[numbers]{natbib}
\newcommand{\jump}[1]{\llbracket #1 \rrbracket}

\journalname{}
\begin{document}

\title{A hybridizable discontinuous Galerkin method for the
  Navier--Stokes equations with pointwise divergence-free velocity
  field\thanks{SR gratefully acknowledges support from the Natural
    Sciences and Engineering Research Council of Canada through the
    Discovery Grant program (RGPIN-05606-2015) and the Discovery
    Accelerator Supplement (RGPAS-478018-2015).} }

\titlerunning{A HDG method for Navier--Stokes with pointwise divergence-free velocity}

\author{Sander Rhebergen \and Garth N.~Wells}


\institute{Sander Rhebergen \at
  Department of Applied Mathematics, University of Waterloo,
  Waterloo N2L~3G1, Canada \\
  \email{srheberg@uwaterloo.ca} \\
  ORCID:~0000-0001-6036-0356 \\
  \and
  Garth N.~Wells \at
  Department of Engineering, University of Cambridge,
  Trumpington Street, Cambridge CB2~1PZ, United Kingdom \\
  \email{gnw20@cam.ac.uk} \\
  ORCID:~0000-0001-5291-7951
}

\date{}

\maketitle

\begin{abstract}
  We introduce a hybridizable discontinuous Galerkin method for the
  incompressible Navier--Stokes equations for which the approximate
  velocity field is pointwise divergence-free. The method builds on
  the method presented by Labeur and Wells [SIAM J. Sci. Comput.,
    vol. 34 (2012), pp. A889--A913]. We show that with modifications
  of the function spaces in the method of Labeur and Wells it is
  possible to formulate a simple method with pointwise divergence-free
  velocity fields which is momentum conserving, energy stable, and
  pressure-robust. Theoretical results are supported by two- and
  three-dimensional numerical examples and for different orders of
  polynomial approximation.

\end{abstract}

\section{Introduction}
\label{sec:introduction}

Numerous finite element methods for the incompressible Navier--Stokes
equations result in approximate velocity fields that are not pointwise
divergence-free. This lack of pointwise satisfaction of the continuity
equation typically leads to violation of conservation laws beyond just
mass conservation, such as conservation of energy. A key issue is
that, in the absence of a pointwise solenoidal velocity field, the
conservative and advective format of the Navier--Stokes equations are
not equivalent.  The review paper by \citet{John:2017} presents cases
for the Stokes limit where the lack of pointwise enforcement of the
continuity equation can lead to large solution errors. Elements that
are stable (in sense of the inf-sup condition), but do not enforce the
continuity equation pointwise, such as the Taylor--Hood,
Crouzeix--Raviart, and MINI elements, can suffer from large errors in
the pressure, which in turn can pollute the velocity
approximation. The concept of `pressure-robustness' to explain the
aforementioned issues is discussed by \citet{John:2017}.  A second
issue is when a computed velocity field that is not pointwise
divergence-free is used as the advective velocity in a transport
solver. The lack of pointwise incompressibility can lead to spurious
results and can compromise stability of the transport equation.

Discontinuous Galerkin (DG) finite element methods provide a natural
framework for handling the advective term in the Navier--Stokes
equations, and have been studied extensively in this context,
e.g.~\citep{Bassi:2006, Cockburn:2005, Pietro:book, Ferrer:2011,
  Rhebergen:2013, Shahbazi:2007}. A difficulty in the construction of
DG methods for the Navier--Stokes equations is that it is not possible
to have both an energy-stable \emph{and} locally momentum conserving
method unless the approximate velocity is exactly
divergence-free~\citep[p.~1068]{Cockburn:2005}. To overcome this
problem, a post-processing operator was introduced by
\citet{Cockburn:2005}. The operator, which is a slight modification of
the Brezzi--Douglas--Marini interpolation operator (see
e.g.~\citep{Boffi:book}), applied to the DG approximate velocity field
generates a post-processed velocity that is pointwise
divergence-free. Key to the operator is that it can be applied
element-wise and is therefore inexpensive to apply. A second issue
with DG methods, and a common criticism, is that the number of
degrees-of-freedom on a given mesh is considerably larger than for a
conforming method. This is especially the case in three spatial
dimensions.

An approach to representing pointwise divergence-free velocity fields
is to use a $H({\rm div})$-conforming velocity field, in which the
normal component of the velocity is continuous across facets, together
with a discontinuous pressure field from an appropriate space.  Such a
velocity space can be constructed by using a $H({\rm div})$-conforming
finite element space, or by enforcing the desired continuity via
hybridization~\citep{Boffi:book}.  However, construction of $H({\rm
  div})$-conforming methods for the Navier--Stokes (and Stokes)
equations is not straightforward as the tangential components of the
viscous stress on cell facets must be appropriately handled. Moreover,
for advection dominated flows it is not immediately clear how the
advective terms can be appropriately stabilized.  Examples of
hybridization for the Stokes equations can be found
in~\citep{Carrero:2005, Cockburn:2005a, Cockburn:2005b}, and for the
Navier--Stokes equations in~\citep{Lehrenfeld:2016}.

A synthesis of discontinuous Galerkin and hybridized methods has lead
to the development of hybridizable Discontinuous Galerkin (HDG) finite
element methods~\citep{Cockburn:2009a,labeur:2007}. These methods were
introduced with the purpose of reducing the computational cost of DG
methods on a given mesh, while retaining the attractive conservation
and stability properties of DG methods.  This is achieved as follows.
The governing equations are posed cell-wise in terms of the
approximate fields on a cell and numerical fluxes, in which the latter
depends on traces of the approximate fields and fields that are
defined only on facets.  Fields defined on a cell are not coupled
directly to fields on neighboring cells, but `communicate' only via
the fields that are defined on facets.  By coupling degrees of freedom
on a cell only to degrees of freedom of the facet functions, cell
degrees of freedom can be eliminated in favor of facet degrees of
freedom only. The result is that the HDG global system of algebraic
equations is significantly smaller than those obtained using~DG.

It has been shown that, after post-processing, solutions obtained by
HDG methods may show super-convergence results for elliptic problems
(for polynomial approximations of order $k$, the order of accuracy is
order $k+2$ in the $L^{2}$-norm). This property has been exploited
also in the context of the Navier--Stokes equations by,
e.g.,~\citep{Cesmelioglu:2016, Nguyen:2011}. Although the velocity
field is not automatically pointwise divergence-free, a
post-processing is applied that results in an approximate velocity
field that is exactly divergence-free and $H({\rm div})$-conforming
and super-converges for low Reynolds number flows.  Super-convergence
is, however, lost when the flow is convection dominated.

We use the HDG approach to construct a simple discretization of the
Navier--Stokes equations in which the computed velocity field is
$H({\rm div})$-conforming and pointwise divergence-free. To achieve
this, we first note that unlike many other HDG methods for
incompressible flows~\citep{Cesmelioglu:2016, Cockburn:2009c,
  Cockburn:2011, Lehrenfeld:2016, Nguyen:2010, Nguyen:2011, Qiu:2016},
the HDG methods of \citet{Labeur:2012} and \citet{Rhebergen:2012}
involve facet unknowns for the pressure.  The pressure field on a cell
plays the role of cell-wise Lagrange multiplier to enforce the
continuity equation, whereas the facet pressure unknowns play the role
of Lagrange multipliers enforcing continuity of the normal component
of the velocity across cell boundaries~\citep{Rhebergen:2017}. It was
shown already in \citep{Labeur:2012} that if the polynomial
approximation of the element pressure on simplices is one order lower
than the polynomial approximation of the velocity that the approximate
velocity field is exactly divergence-free on cells.  However, the
method in \citep{Labeur:2012} could not simultaneously satisfy mass
conservation, momentum conservation and energy stability. This
shortcoming is due to the computed velocity field for the method in
\citep{Labeur:2012} not being $H({\rm div})$-conforming.  We note that
fast solvers for the Stokes part of the problem are developed and
analysed in~\citep{rhebergen:2018}.

In this paper we show that if the facet pressure space is chosen
appropriately, we obtain approximate velocity fields that are
$H({\rm div})$-conforming and pointwise divergence-free. We are guided
in this by the stability analysis in \citep{Rhebergen:2017} for the
Stokes problem, which provides guidance on the permissible function
spaces. The consequences of this modification of the method
of~\citep{Labeur:2012} are profound: the method proposed in this work
results in a scheme that is both mass and momentum conserving (locally
and globally), energy stable and pressure-robust. We summarize
properties of the proposed method and those of~\citep{Labeur:2012} in
\cref{tab:summary_properties}.

\begin{table}
  \caption{Summary of the properties of the method
    of~\citep{Labeur:2012} and the proposed method of this paper. The
    skew-symmetric and divergence forms refer to different
    formulations of the momentum equation.  In~\citep{Labeur:2012}
    both an equal- and mixed-order velocity-pressure approximation are
    introduced.}
  \label{tab:summary_properties}
  \centering {\small
    \begin{tabular}{c|cccc}
      \hline
      Formulation & mass & momentum & energy & pressure \\
      & conserving & conserving & stable & robust \\
      \hline
      Equal order~\citep{Labeur:2012} & $\times$     & $\checkmark$ & only in        & $\times$ \\
                                      &              &              & skew-symmetric &          \\
                                      &              &              & form           &          \\
      \hline
      Mixed order~\citep{Labeur:2012} & $\checkmark$ & only in      & only in        & $\times$ \\
                                      &              & divergence   & skew-symmetric &          \\
                                      &              & form         & form           &          \\
      \hline
      Proposed method                 & $\checkmark$ & $\checkmark$ & $\checkmark$   & $\checkmark$ \\
      \hline
    \end{tabular}
  }
\end{table}

The remainder of this paper is organized as
follows. \Cref{sec:navierstokes} briefly introduces the Navier--Stokes
problem, which is followed by the main result of this paper in
\cref{sec:hdg}; a momentum conserving and energy stable HDG method for
the Navier--Stokes equations with pointwise solenoidal and $H({\rm
  div})$-conforming velocity field. Numerical results are presented in
\cref{sec:compres} and conclusions are drawn in
\cref{sec:conclusions}.

\section{Incompressible Navier--Stokes problem}
\label{sec:navierstokes}

Let $\Omega \subset \mathbb{R}^d$ be a polygonal ($d = 2$) or
polyhedral ($d = 3$) domain with boundary outward unit normal $n$, and
let the time interval of interest be given by $I = (0,t_N]$. Given the
  kinematic viscosity $\nu \in \mathbb{R}^+$ and forcing term $f :
  \Omega \times I \to \mathbb{R}^d$, the Navier--Stokes equations for
  the velocity field $u : \Omega \times I \to \mathbb{R}^d$ and
  kinematic pressure field $p : \Omega \times I \to \mathbb{R}$ are
  given by
\begin{subequations}
  \label{eq:navsto}
  \begin{align}
    \label{eq:navsto_mom}
    \partial_t u + \nabla \cdot \sigma &= f && \mbox{in} \ \Omega \times I,
    \\
    \label{eq:navsto_mass}
    \nabla \cdot u &= 0 && \mbox{in} \ \Omega \times I,
  \end{align}
\end{subequations}
where $\sigma$ is the momentum flux:
\begin{equation}
  \label{eq:momentumflux}
  \sigma := \sigma_a +\sigma_d
  \quad \mbox{with} \quad
  \sigma_a := u \otimes u
  \quad \mbox{and} \quad
  \sigma_d := p \mathbb{I} - \nu \nabla u,
\end{equation}
and $\mathbb{I}$ is the identity tensor and $(a \otimes b)_{ij} = a_i
b_j$.

We partition the boundary of $\Omega$ such that $\partial \Omega =
\Gamma_D \cup \Gamma_N$ and $\Gamma_D \cap \Gamma_N = \emptyset$.
Given $h : \Gamma_N \times I \to \mathbb{R}^d$ and a solenoidal
initial velocity field $u_0 : \Omega \to \mathbb{R}^d$, we prescribe
the following boundary and initial conditions:
\begin{subequations}
  \label{eq:bc_ic}
  \begin{align}
    \label{eq:bc_dirichlet}
    u &= 0 && \mbox{on} \ \Gamma_D \times I,
    \\
    \label{eq:bc_neumann}
    \sigma \cdot n - \max\del{u \cdot n, 0}{u} &= h && \mbox{on}
    \ \Gamma_N\times I, \\
    \label{eq:ic}
    u(x, 0) &= u_0(x) && \mbox{in} \ \Omega.
  \end{align}
\end{subequations}
On inflow parts of $\Gamma_{N}$ ($u \cdot n < 0$) we impose the total
momentum flux, i.e., $\sigma \cdot n = h$. On outflow parts of
$\Gamma_N$ ($ u \cdot n \ge 0$), only the diffusive part of the
momentum flux is prescribed, i.e.,~$\sigma_d \cdot n = h$.

\Cref{eq:navsto_mom} is the conservative form of the Navier--Stokes
equation.  With satisfaction of the incompressibility constraint,
\cref{eq:navsto_mass}, the momentum equation \eqref{eq:navsto_mom} can
be equivalently expressed as:
\begin{equation}
  \label{eq:ns_mom_weighted}
  \partial_t u + (1 - \chi) u \cdot \nabla u + \chi \nabla \cdot
  \sigma_a + \nabla \cdot \sigma_d = f,
\end{equation}
where $\chi \in [0, 1]$.  For numerous finite element methods, the
approximate velocity field is not pointwise or locally (in a weak
sense) solenoidal.  In such cases, it can be shown that momentum is
conserved if $\chi = 1$, while energy stability can be proven if~$\chi
= 1/2$. For stabilized finite element methods in which the continuity
equation is not satisfied locally, manipulations of the advective term
can be applied to achieve momentum conservation~\citep{hughes:2005}.

The mass conserving (mixed-order) hybridizable discontinuous Galerkin
method of \citet{Labeur:2012} is based on a weak formulation of
\cref{eq:ns_mom_weighted}. It was proven to be locally momentum
conserving for $\chi = 1$ and energy stable for $\chi = 1/2$, but in
their analysis both properties could not be satisfied simultaneously.
We will prove how the method can be formulated such that mass and
momentum conservation, and energy stability can be satisfied
simultaneously, and the method be made invariant with respect
to~$\chi$.

\section{A hybridizable discontinuous Galerkin method}
\label{sec:hdg}

We present a hybridizable discontinuous Galerkin method for the
Navier--Stokes problem for which the approximate velocity field is
pointwise divergence-free.

\subsection{Preliminaries}

Let $\mathcal{T} := \cbr{K}$ be a triangulation of the domain $\Omega$
into non-overlapping simplex cells~$K$. The boundary of a cell is
denoted by $\partial K$ and the outward unit normal vector on
$\partial K$ by~$n$. Two adjacent cells $K^+$ and $K^-$ share an
interior facet $F := \partial K^+ \cap \partial K^-$. A facet of
$\partial K$ that lies on the boundary of the domain $\partial \Omega$
is called a boundary facet. The sets of interior and boundary facets
are denoted by $\mathcal{F}_I$ and $\mathcal{F}_B$, respectively. The
set of all facets is denoted by
$\mathcal{F} := \mathcal{F}_I \cup \mathcal{F}_B$.

\subsection{Semi-discrete formulation}
\label{ss:semidisc_wf_new}

Consider the following finite element spaces:
\begin{subequations}
  \label{eq:general_func_spc}
  \begin{align}
  \label{eq:general_func_spc_a}
    V_h &:= \cbr{ v_h \in \sbr{L^2(\mathcal{T})}^d, \
      v_h \in \sbr{P_k(K)}^d\ \forall K \in \mathcal{T}},
    \\
    \label{eq:general_func_spc_b}
    \bar{V}_h &:= \cbr{ \bar{v}_h \in \sbr{L^2(\mathcal{F})}^d, \
      \bar{v}_h \in \sbr{P_k(F)}^d
      \ \forall F \in \mathcal{F},\ \bar{v}_h = 0 \ \mbox{on}\ \Gamma_D},
    \\
    \label{eq:general_func_spc_c}
    Q_h &:= \cbr{ q_h \in L^2(\mathcal{T}),\ q_h \in P_{k-1}(K)\
      \forall K\in\mathcal{T}},
    \\
    \label{eq:general_func_spc_d}
    \bar{Q}_h &:= \cbr{ \bar{q}_h \in L^2(\mathcal{F}),\
      \bar{q}_h \in P_{k}(F) \ \forall F\in\mathcal{F}},
  \end{align}
\end{subequations}
where $P_l(D)$ denotes the space of polynomials of degree $l > 0$ on a
domain~$D$. Note that the spaces $V_{h}$ and $Q_{h}$ are defined on
the whole domain $\mathcal{T}$, whereas the spaces $\Bar{V}_{h}$ and
$\Bar{Q}_{h}$ are defined only on facets of the triangulation.

The spaces $V_h$ and $Q_h$ are discontinuous across cell boundaries,
hence the trace of a function $a \in V_h$ may be double-valued on cell
boundaries. At an interior facet, $F$, we denote the traces of
$a \in V_{h}$ by $a^+$ and~$a^-$. We introduce the jump operator
$\jump{a} := a^+ \cdot n^+ + a^- \cdot n^-$, where $n^{\pm}$ the
outward unit normal on~$\partial K^{\pm}$.

We now state the weak formulation of the proposed method: given a
forcing term $f \in \sbr{L^2(\Omega)}^d$, boundary condition $h \in
\sbr{L^2(\Gamma_N)}^d$ and viscosity $\nu$, find $u_h, \bar{u}_h,
p_h,\bar{p}_h \in V_h \times \bar{V}_h \times Q_h \times \bar{Q}_h$
such that:
\begin{subequations}
  \label{eq:wf_simp}
  \begin{align}
    \label{eq:mass_loc_simp}
    0 =& \sum_K\int_K u_h \cdot \nabla q_h \dif x
    - \sum_K \int_{\partial K} u_h\cdot n \, q_h \dif s
    \quad \forall q_{h} \in Q_{h},
    \\
    \label{eq:mass_glob_simp}
    0 =& \sum_K\int_{\partial K} u_h\cdot n \, \bar{q}_h \dif s
    - \int_{\partial \Omega} \bar{u}_h\cdot n \, \bar{q}_h \dif s
    \quad \forall \Bar{q}_{h} \in \Bar{Q}_{h},
  \end{align}
  and
  \begin{multline}
    \label{eq:mom_loc_simp}
    \int_{\Omega} f \cdot v_h \dif x
    =
    \int_{\Omega} \partial_t u_h \cdot v _h \dif x
    -\sum_K \int_K \sigma_{h} : \nabla v_h \dif x
    +\sum_K \int_{\partial K} \hat{\sigma}_{h} : \del{v_h\otimes n} \dif s
\\
    + \sum_K\int_{\partial K}\nu \del{\del{\bar{u}_h - u_h} \otimes n} : \nabla v_h\dif s
    \quad \forall v_{h} \in V_{h},
  \end{multline}
  \begin{equation}
  \label{eq:mom_glob_simp}
    \int_{\Gamma_N} h \cdot\bar{v}_h\dif s
    = \sum_K\int_{\partial K} \hat{\sigma}_{h} : \del{\bar{v}_h \otimes n}  \dif s
    - \int_{\Gamma_N} \del{1 - \lambda}\del{\bar{u}_h \cdot n} \bar{u}_h \cdot \bar{v}_h\dif s
    \quad \forall \Bar{v}_{h} \in \Bar{V}_{h},
  \end{equation}
\end{subequations}
where $\hat{\sigma}_h := \hat{\sigma}_{a,h} + \hat{\sigma}_{d,h}$ is
the `numerical flux' on cell facets. The advective part of the
numerical flux is given by:
\begin{equation}
  \label{eq:numflux_sigmaa}
  \hat{\sigma}_{a,h} := \sigma_{a,h} + \del{\bar{u}_h - u_h} \otimes \lambda u_h,
\end{equation}
where $\lambda$ is an indicator function that takes on a value of
unity on inflow cell boundaries (where $u_h \cdot n < 0$) and a value
of zero on outflow cell facets (where $u_h \cdot n \ge 0$).  This
definition of the numerical flux provides upwinding of the advective
component of the flux.  The diffusive part of the numerical flux is
defined as
\begin{equation}
  \label{eq:numflux_sigmad}
  \hat{\sigma}_{d,h} := \bar{p}_h \mathbb{I} - \nu\nabla u_h
  - \frac{\nu \alpha}{h_K} \del{\bar{u}_h - u_h}\otimes n,
\end{equation}
where $\alpha > 0$ is a penalty parameter as is typical of Nitsche and
interior penalty methods. It is proven in
\citep{Wells:2011,Rhebergen:2017} that $\alpha$ needs to be
sufficiently large to ensure stability.

A key feature of this formulation, and what distinguishes it from
standard discontinuous Galerkin methods, is that functions on cells
(functions in $V_h$ and $Q_h$) are not coupled across facets directly
via the numerical flux. Rather, fields on neighboring cells are
coupled via the facet functions $\Bar{u}_{h}$ and~$\Bar{p}_{h}$. The
fields $u_{h}$ and $p_{h}$ can therefore be eliminated locally via
static condensation, resulting in a global system of equations in
terms of the facet functions only.  This substantially reduces the
size of the global systems compared to a standard discontinuous
Galerkin method on the same mesh, yet still permits the natural
incorporation of upwinding and cell-wise balances.

The weak formulation presented here is the weak formulation of
\citet{Labeur:2012} with conservative form of the advection term
($\chi = 1$ in \cref{eq:ns_mom_weighted}).  The key difference is that
we have been more prescriptive on the relationships between the finite
element spaces in \cref{eq:general_func_spc}, and we will prove that
this leads to some appealing properties. In particular, the spaces in
\cref{eq:general_func_spc} are such that: for $u_h \in
\sbr{P_k(K)}^d$, $\nabla \cdot u_h \in P_{k-1}(K)$ and $u_h \cdot n
\in P_{k}(F)$; and for $\bar{u}_h \in \sbr{P_{k}(F)}^d$,
$\bar{u}_h\cdot n \in P_{k}(F)$. Furthermore, the function spaces have
been chosen such that the resulting method is inf-sup stable,
see~\citep{Rhebergen:2017}. The resulting weak formulation can be
shown to be equivalent to a weak formulation in which the approximate
velocity field lies in the Brezzi--Douglas--Marini (BDM) finite
element space~\citep[Section~3.4]{Rhebergen:2017}. Hybridization of
other $H({\rm div})$ conforming finite element spaces, see
e.g.~\citep{Boffi:book}, are also possible.

\begin{proposition}[mass conservation]
  \label{prop:masscons}
  If $u_h \in V_h$ and $\bar{u}_h \in \bar{V}_h$ satisfy
  \cref{eq:wf_simp}, with $V_h$ and $\bar{V}_h$ defined in
  \cref{eq:general_func_spc}, then
  \begin{equation}
    \label{eq:divuzero}
    \nabla \cdot u_h = 0 \qquad \forall x \in K,\ \forall K \in \mathcal{T},
  \end{equation}
  and
  \begin{subequations}
    \label{eq:unbarun}
    \begin{align}
      \label{eq:unbarun_I}
      \jump{u_h} &= 0 && \forall x \in F,\
      \forall F\in\mathcal{F}_I,
      \\
      \label{eq:unbarun_B}
      u_h \cdot n &= \bar{u}_h \cdot n  && \forall x \in F,\
    \forall F\in\mathcal{F}_B.
    \end{align}
  \end{subequations}
\end{proposition}
\begin{proof}
  Applying integration-by-parts to \cref{eq:mass_loc_simp}:
  \begin{equation}
    0
    = \int_K q_h \nabla \cdot u_h \dif x \qquad
    \forall q_{h} \in P_{k-1}(K), \
    \forall K \in \mathcal{T}.
  \end{equation}
  Since $q_{h}$, $\nabla \cdot u_h \in P_{k-1}(K)$, pointwise
  satisfaction of the continuity equation, \cref{eq:divuzero},
  follows.

  It follows from \cref{eq:mass_glob_simp} that:
  \begin{equation}
    \label{eq:mass_glob_inproof}
    0
    =
    \sum_{F\in\mathcal{F}_I} \int_{F} \jump{u_h} \bar{q}_h \dif s
    + \sum_{F\in\mathcal{F}_B} \int_{F} \del{u_h - \bar{u}_h } \cdot n \bar{q}_h \dif s
    \quad \forall \Bar{q}_{h} \in \Bar{Q}_{h}.
  \end{equation}
  Since $\bar{q}_{h}$, $u_h \cdot n$, $\bar{u}_h \cdot n \in
  P_{k}(F)$, \cref{eq:unbarun} follows. \qed
\end{proof}

\Cref{prop:masscons} is a stronger statement of mass conservation than
in \citet[Proposition~4.2]{Labeur:2012}, in which mass conservation
for the mixed-order case was proved locally (cell-wise) in an integral
sense only.  Under certain conditions, implementations in
\citep{Labeur:2012} satisfy \cref{eq:divuzero}, but not
\cref{eq:unbarun}.  We will show that this difference is critical for
the formulation in this work as it allows simultaneous satisfaction of
momentum conservation and energy stability.

We next show momentum conservation for the semi-discrete weak
formulation in terms of the numerical flux.
\begin{proposition}[momentum conservation]
  \label{prop:momcons}
  Let $u_h, \bar{u}_h, p_h, \bar{p}_h \in V_h \times \bar{V}_h \times
  Q_h \times \bar{Q}_h$ satisfy \cref{eq:wf_simp}. Then,
  \begin{equation}
    \label{eq:momconsloc}
    \od{}{t}\int_K u_h \dif x = \int_K f \dif x
    - \int_{\partial K}\hat{\sigma}_h n \dif s
    \quad \forall K \in \mathcal{T}.
  \end{equation}
  Furthermore, if $\Gamma_D = \emptyset$,
  \begin{equation}
    \label{eq:momconsglob}
    \od{}{t} \int_{\Omega} u_h \dif x
    = \int_{\Omega} f \dif x
    - \int_{\partial\Omega}(1-\lambda)(\bar{u}_h \cdot n) \bar{u}_h \dif s
    - \int_{\partial\Omega} h \dif s.
  \end{equation}
\end{proposition}
\begin{proof}
  In \cref{eq:mom_loc_simp}, set $v_h = e_j$ on $K$, where $e_j$ is a
  canonical unit basis vector, and set $v_h = 0$ on $\mathcal{T}
  \backslash K$ in \cref{eq:mom_loc_simp}:
  \begin{equation}
    \od{}{t}\int_{K} u_h\cdot e_j \dif x
    +\int_{\partial K} \del{\hat{\sigma}_h \cdot n} \cdot e_j \dif s
    = \int_{K} f \cdot e_{j} \dif x,
  \end{equation}
  which proves \cref{eq:momconsloc}. \Cref{eq:momconsglob} follows
  immediately by setting $v_h = e_j$ in \cref{eq:mom_loc_simp},
  $\bar{v}_h = - e_j$ in \cref{eq:mom_glob_simp} and summing the two
  results. \qed
\end{proof}

We next prove that the method is \emph{also} globally energy stable.

\begin{proposition}[global energy stability]
  \label{prop:globenerStab}
  If $u_h, \bar{u}_h, p_h, \bar{p}_h \in V_h \times \bar{V}_h \times
  Q_h \times \bar{Q}_h$ satisfy \cref{eq:wf_simp}, for homogeneous
  boundary conditions, $f = 0$ and for a suitably large~$\alpha$:
  \begin{equation}
    \label{eq:globenerStab}
    \od{}{t}\sum_K\int_K \envert{u_h}^2 \dif x \le 0.
  \end{equation}
\end{proposition}
\begin{proof}
  Setting $q_h = -p_h$, $\bar{q}_h = -\bar{p}_h$, $v_h = u_h$ and
  $\bar{v}_h = -\bar{u}_h$ in
  \cref{eq:mass_loc_simp,eq:mass_glob_simp,eq:mom_loc_simp,eq:mom_glob_simp}
  and inserting the expressions for the numerical fluxes
  (\cref{eq:momentumflux,eq:numflux_sigmaa,eq:numflux_sigmad}),
  and summing:
  \begin{multline}
    \label{eq:afterPresIntbyparts}
    \sum_K \frac{1}{2} \int_K \partial_t \envert{u_h}^2 \dif x
    + \sum_K \frac{1}{2} \int_{\partial K} \del{u_{h} \cdot n} \envert{u_h}^{2} \dif s
    \\
    - \sum_K \frac{1}{2} \int_{\partial K} \del{u_{h} \cdot n} \envert{\bar{u}_h}^{2} \dif s
    + \sum_K \frac{1}{2} \int_{\partial K} \envert{u_h \cdot n} \envert{u_h - \bar{u}_h}^2 \dif s
    \\
    + \sum_K \int_K \nu \envert{\nabla u_h}^2 \dif x
    + \sum_K \int_{\partial K} \frac{\nu \alpha}{h_K} \envert{\bar{u}_h - u_h}^2 \dif s
    \\
    + 2\sum_K \int_{\partial K} \nu \del{\nabla u_h \cdot n} \cdot \del{\bar{u}_h - u_h} \dif s
    + \int_{\Gamma_N} (1-\lambda)(\bar{u}_h \cdot n) \envert{\bar{u}_h}^2 \dif s
    \\
    - \sum_K \int_K \del{u_h \otimes u_h} : \nabla u_h \dif x
    = 0,
  \end{multline}
  where we have used that $\lambda u_{h} \cdot n = \del{u_{h} \cdot n
    - |u_{h} \cdot n|}/2$, and applied integration-by-parts to the
  pressure gradient terms.  Since $\bar{u}_h$ is single-valued on
  facets, the normal component of $u_h$ is continuous across facets
  and $\bar{u}_h \cdot n = u_{h} \cdot n$ on the domain boundary (see
  \cref{prop:masscons}), the third integral on the left-hand side of
  \cref{eq:afterPresIntbyparts} can be simplified:
  \begin{equation}
    \label{eq:intzeroterm}
    - \sum_K \frac{1}{2} \int_{\partial K} \del{u_{h} \cdot n} \envert{\bar{u}_h}^{2} \dif s
    =
      - \frac{1}{2} \int_{\Gamma_N} \del{\bar{u}_h \cdot n} \envert{\bar{u}_h}^{2} \dif s.
  \end{equation}
  We consider now the last term on the left-hand side of
  \cref{eq:afterPresIntbyparts}. On each cell $K$ it holds that $-u_h
  \otimes u_h : \nabla u_h = (\nabla \cdot u_h)(u_h \cdot u_h)/2 -
  \nabla\cdot((u_h \otimes u_h) \cdot u_h)/2 = - \nabla \cdot((u_h
  \otimes u_h) \cdot u_h)/2$, since $\nabla \cdot u_h = 0$ (by
  \cref{prop:masscons}). It follows that
  \begin{equation}
    \label{eq:rewritelast}
    - \sum_K \int_K \del{u_h \otimes u_h} : \nabla u_h \dif x
    = -\frac{1}{2} \sum_K \int_{\partial K} \del{u_{h} \cdot n} \envert{u_{h}}^{2} \dif s.
  \end{equation}
  Combining \cref{eq:afterPresIntbyparts,eq:intzeroterm,eq:rewritelast},
  \begin{multline}
    \label{eq:afterPresIntbyparts_simp}
    \frac{1}{2} \sum_K  \int_K  \partial_t \envert{u_h}^2 \dif x  =
    -\frac{1}{2}\sum_K  \int_{\partial K}  \envert{u_h \cdot n} \envert{u_h - \bar{u}_h}^2 \dif s
    \\
    -\sum_K \int_K \nu \envert{\nabla u_h}^2 \dif x
    - \sum_K \int_{\partial K} \frac{\nu \alpha}{h_K} \envert{\bar{u}_h - u_h}^2 \dif s
    \\
    - 2 \sum_K \int_{\partial K} \nu \del{\nabla u_h n} \cdot \del{\bar{u}_h - u_h} \dif s
    - \frac{1}{2} \int_{\Gamma_N} \envert{\bar{u}_h \cdot n} \envert{\bar{u}_h}^2 \dif s,
  \end{multline}
  where we have used that
  \begin{equation}
    \int_{\Gamma_N} (1 - \lambda)(\bar{u}_h \cdot n) \envert{\bar{u}_h}^2 \dif s
    - \frac{1}{2} \int_{\Gamma_N} \del{\bar{u}_h \cdot n} \envert{\bar{u}_h}^{2} \dif s
    =
    \frac{1}{2} \int_{\Gamma_N} \envert{\bar{u}_h \cdot n} \envert{\bar{u}_h}^2 \dif s.
  \end{equation}
  It can be proven that there exists an $\alpha > 0$, independent of
  $h_K$, such that
  \begin{multline}
    \sum_K \int_K \nu \envert{\nabla u_h}^2 \dif x +
    \sum_K \int_{\partial K} \frac{\nu \alpha}{h_K} \envert{\bar{u}_h - u_h}^2 \dif s
    \\ \ge
    2 \envert{\sum_K \int_{\partial K} \nu \del{\nabla u_h \cdot n} \cdot \del{\bar{u}_h - u_h} \dif s},
  \end{multline}
  (see~\cite[Lemma~5.2]{Wells:2011}
  and~\cite[Lemma~4.2]{Rhebergen:2017}).  Therefore, the right-hand
  side of \cref{eq:afterPresIntbyparts_simp} is non-positive, proving
  \cref{eq:globenerStab}. \qed
\end{proof}

The key results that enable us to prove global energy stability for
this conservative form of the Navier--Stokes equations are: (a) the
pointwise solenoidal velocity field; and (b) continuity of the normal
component of the velocity field across facets. The latter point is not
fulfilled by the method in~\citep{Labeur:2012}.

\subsection{A fully-discrete weak formulation}
\label{ss:disc_wf_new}

We now consider a fully-discrete formulation. We partition the time
interval $I$ into an ordered series of time levels $0 = t^0 < t^1 <
\cdots < t^N$. The difference between each time level is denoted by
$\Delta t^n = t^{n+1} - t^{n}$. To discretize in time, we consider the
$\theta$-method and denote midpoint values of a function $y$ by $y^{n
  + \theta} := (1-\theta)y^n + \theta y^{n+1}$.  Following
\citet{Labeur:2012}, the convective velocity will be evaluated at the
current time $t^n$, thereby linearizing the problem,~i.e.:
\begin{equation}
  \sigma_{h}^{n + \theta} = \sigma_{a,h}^{n+\theta} + \sigma_{d,h}^{n + \theta}
  \qquad \mbox{where} \qquad
  \sigma_{a,h}^{n + \theta} = u_h^{n+\theta} \otimes u_h^n,
\end{equation}
and
\begin{equation}
  \hat{\sigma}_{h}^{n + \theta} = \hat{\sigma}_{a,h}^{n + \theta}
  + \hat{\sigma}_{d,h}^{n+\theta}
  \qquad \mbox{where} \qquad
  \hat{\sigma}_{a,h}^{n+\theta}
  =
  \sigma_{a,h}^{n + \theta}
  + (\bar{u}_h^{n+\theta} - u_h^{n+\theta}) \otimes \lambda u_h^n.
\end{equation}
The time-discrete counterpart of \cref{eq:wf_simp} is: given $u_{h}^n,
\bar{u}_{h}^n, p_{h}^n, \bar{p}_{h}^n \in V_h \times \bar{V}_h \times
Q_h \times \bar{Q}_h$ at time $t^n$, the forcing term $f^{n+\theta}
\in \sbr{L^2(\Omega)}^d$, the boundary condition $h^{n+\theta} \in
\sbr{L^2(\Gamma_N)}^d$, and the viscosity $\nu$, find $u_{h}^{n+1},
\bar{u}_{h}^{n+1}, p_{h}^{n+1}, \bar{p}_{h}^{n+1} \in V_h \times
\bar{V}_h \times Q_h \times \bar{Q}_h$ such that mass conservation,
\begin{subequations}
  \label{eq:wf_simp_t}
  \begin{align}
    \label{eq:mass_loc_simp_t}
    0 =& \sum_K \int_K u_{h}^{n+1} \cdot \nabla q_h \dif x
    - \sum_K\int_{\partial K} u_{h}^{n+1} \cdot n \, q_h \dif s,
    \\
    \label{eq:mass_glob_simp_t}
    0 =& \sum_K\int_{\partial K} u_{h}^{n+1} \cdot n \, \bar{q}_h \dif s
         - \int_{\partial \Omega} \bar{u}_{h}^{n+1} \cdot n \, \bar{q}_h \dif s,
  \end{align}
  and momentum conservation,
  \begin{multline}
    \label{eq:mom_loc_simp_t}
    \int_{\Omega} f^{n + \theta} \cdot v_h \dif x
    = \int_{\Omega} \frac{u_{h}^{n+1} - u_{h}^{n}}{\Delta t^n} \cdot v_h \dif x
    -\sum_K \int_K \sigma_{h}^{n+\theta} : \nabla v_h \dif x
    \\
    +\sum_K \int_{\partial K} \hat{\sigma}_{h}^{n+\theta} : v_h\otimes n \dif s
    \\
    + \sum_K \int_{\partial K} \nu \del{\del{\bar{u}_{h}^{n+\theta} - u_{h}^{n+\theta}} \otimes n} : \nabla v_h\dif s,
  \end{multline}
  \begin{multline}
    \label{eq:mom_glob_simp_t}
    \int_{\Gamma_N} h^{n+\theta} \cdot \bar{v}_h\dif s
    = \sum_K \int_{\partial K} \hat{\sigma}_{h}^{n+\theta} : \bar{v}_h \otimes n \dif s
    \\
    - \int_{\Gamma_N} \del{1 - \lambda} \del{\bar{u}_{h}^{n} \cdot n} \bar{u}_{h}^{n+\theta} \cdot \bar{v}_h\dif s,
  \end{multline}
\end{subequations}
are satisfied for all $v_h, \bar{v}_h, q_h, \bar{q}_h \in V_h \times
\bar{V}_h \times Q_h \times \bar{Q}_h$.  Here $\lambda$ is evaluated
using the known velocity field at time~$t^n$.

In \cref{ss:semidisc_wf_new} we proved that the semi-discrete
formulation \cref{eq:wf_simp} is momentum conserving, energy stable
and exactly mass conserving when using the function spaces given by
\cref{eq:general_func_spc}. We show next that the fully-discrete
formulation given by \cref{eq:wf_simp_t} inherits these properties.

\begin{proposition}[fully-discrete mass conservation]
  \label{prop:masscons_t}
  If $u_{h}^{n+1} \in V_h$ and $\bar{u}_{h}^{n+1} \in \bar{V}_h$
  satisfy \cref{eq:wf_simp_t}, then
  \begin{equation}
    \label{eq:divuzero_t}
    \nabla \cdot u_{h}^{n+1} = 0 \qquad \forall x \in K,\
    \forall K\in\mathcal{T},
  \end{equation}
  and
  \begin{subequations}
    \label{eq:unbarun_t}
    \begin{align}
      \jump{u_{h}^{n+1}} &= 0 && \forall x \in \mathcal{F},\
      \forall \mathcal{F} \in \mathcal{F}_I,
      \\
      u_{h}^{n+1} \cdot n &= \bar{u}_{h}^{n+1} \cdot n  &&
      \forall x \in \mathcal{F},\ \forall \mathcal{F} \in \mathcal{F}_B.
    \end{align}
  \end{subequations}
\end{proposition}
\begin{proof}
  The proof is similar to that of \cref{prop:masscons} and therefore
  omitted. \qed
\end{proof}

\begin{proposition}[fully-discrete momentum conservation]
  \label{prop:momcons_t}
  If
  $u_{h}^{n}, \bar{u}_{h}^{n}, p_{h}^{n}, \bar{p}_{h}^{n} \in V_h
  \times \bar{V}_h \times Q_h \times \bar{Q}_h$
  and
  $u_{h}^{n+1}, \bar{u}_{h}^{n+1}, p_{h}^{n+1}, \bar{p}_{h}^{n+1} \in
  V_h \times \bar{V}_h \times Q_h \times \bar{Q}_h$
  satisfy \cref{eq:wf_simp_t}, then
  \begin{equation}
    \label{eq:momconsloc_t}
    \int_K \frac{u_{h}^{n+1} - u_{h}^{n}} {\Delta t^{n}} \dif x
    = \int_K f^{n + \theta} \dif x
    - \int_{\partial K}\hat{\sigma}_{h}^{n + \theta} n \dif s
    \quad \forall K \in \mathcal{T}.
  \end{equation}
  Furthermore, if $\Gamma_D = \emptyset$,
  \begin{multline}
    \label{eq:momconsglob_t}
    \sum_K \int_K \frac{u_{h}^{n + 1} - u_{h}^{n}} {\Delta t^{n}} \dif x
    = \sum_K \int_K f^{n+\theta} \dif x
    - \int_{\partial\Omega} (1 - \lambda)(\bar{u}_{h}^{n} \cdot n) \bar{u}_{h}^{n + \theta} \dif s \\
    - \int_{\partial\Omega} h^{n + \theta} \dif s.
  \end{multline}
\end{proposition}
\begin{proof}
  The proof is similar to that of \cref{prop:momcons} and therefore
  omitted. \qed
\end{proof}

\begin{proposition}[fully-discrete energy stability]
  \label{prop:globenerStab_t}
  If
  $u_{h}^{n}, \bar{u}_{h}^{n}, p_{h}^{n}, \bar{p}_{h}^{n} \in V_h
  \times \bar{V}_h \times Q_h \times \bar{Q}_h$
  and
  $u_{h}^{n+1}, \bar{u}_{h}^{n+1}, p_{h}^{n+1}, \bar{p}_{h}^{n+1} \in
  V_h \times \bar{V}_h \times Q_h \times \bar{Q}_h$
  satisfy \cref{eq:wf_simp_t}, then with homogeneous boundary
  conditions, no forcing terms, for suitably large $\alpha$, and
  $\theta \ge 1/2$,
  \begin{equation}
    \label{eq:globenerStab_t}
    \sum_K \int_K \envert{u_{h}^{n+1}}^2 \dif x
    \le \sum_K \int_K \envert{u_{h}^{n}}^2 \dif x.
  \end{equation}
\end{proposition}
\begin{proof}
  Setting $q_h = -\theta p_{h}^{n+\theta}$, $\bar{q}_h =
  -\theta\bar{p}_{h}^{n + \theta}$, $v_h = u _{h}^{n + \theta}$ and
  $\bar{v}_h = -\bar{u}_{h}^{n + \theta}$, in
  \cref{eq:mass_loc_simp_t,eq:mass_glob_simp_t,eq:mom_loc_simp_t,eq:mom_glob_simp_t},
  adding the results, using the expressions for the diffusive fluxes,
  given by \cref{eq:momentumflux,eq:numflux_sigmad}, partial
  integration of the pressure gradient terms and using that $\nabla
  \cdot u_{h}^{n} = 0$ by \cref{prop:masscons_t}, we obtain, using the
  same steps as in the proof of \cref{prop:globenerStab},
  \begin{multline}
    \label{eq:afterPresIntbyparts_simp_t}
    \int_{\Omega} \frac{u_{h}^{n+1} - u_{h}^{n}}{\Delta t^n} \cdot u_{h}^{n+\theta} \dif x
    + \sum_K \frac{1}{2} \int_{\partial K} \envert{u_{h}^{n} \cdot n} \envert{u_{h}^{n+\theta} - \bar{u}_{h}^{n+\theta}}^2 \dif s
    \\
    +\sum_K \int_K \nu \envert{\nabla u_{h}^{n+\theta}}^2 \dif x
    +\sum_K \int_{\partial K} \frac{\nu \alpha}{h_K} \envert{\bar{u}_{h}^{n+\theta} - u_{h}^{n+\theta}}^2 \dif s
    \\
    + 2\sum_K \int_{\partial K} \nu \del{\nabla u_{h}^{n+\theta} \cdot n}\del{\bar{u}_{h}^{n+\theta} - u_{h}^{n+\theta}} \dif s
\\
    + \frac{1}{2} \int_{\Gamma_N} \envert{\bar{u}_{h}^{n} \cdot n} \envert{ \bar{u}_{h}^{n+\theta}}^2 \dif s
    = 0.
  \end{multline}
  The first term on the left-hand side of
  \cref{eq:afterPresIntbyparts_simp_t} can be reformulated as
  \begin{multline}
    \int_{\Omega} \frac{u_{h}^{n+1} - u_{h}^{n}}{\Delta t^n} \cdot u_{h}^{n+\theta} \dif x =
    \del{\theta - \frac{1}{2}} \int_{\Omega} \frac{\envert{u_{h}^{n+1} - u_{h}^{n}}^2}{\Delta t^n} \dif x
    \\
    + \frac{1}{2} \int_{\Omega} \frac{\envert{u_{h}^{n+1}}^2}{\Delta t^n} \dif x
    - \frac{1}{2} \int_{\Omega} \frac{\envert{u_{h}^{n}}^2}{\Delta t^n} \dif x.
  \end{multline}
  Inserting this expression into \cref{eq:afterPresIntbyparts_simp_t}:
  \begin{multline}
    \label{eq:afterPresIntbypartsDiscrete_simp_tf}
    \frac{1}{2} \int_{\Omega} \frac{\envert{u_{h}^{n+1}}^2}{\Delta t^n} \dif x
    - \frac{1}{2} \int_{\Omega} \frac{\envert{u_{h}^{n}}^2}{\Delta t^n} \dif x
    =
    -\del{\theta - \frac{1}{2}}\int_{\Omega} \frac{\envert{u_{h}^{n+1} - u_{h}^{n}}^2}{\Delta t^n} \dif x
    \\
    -\sum_K \frac{1}{2} \int_{\partial K} \envert{u_{h}^{n} \cdot n} \envert{u_{h}^{n+\theta} - \bar{u}_{h}^{n+\theta}}^2 \dif s
\\
    - 2\sum_K \int_{\partial K} \nu \del{\nabla u_{h}^{n+\theta} \cdot n}\del{\bar{u}_{h}^{n+\theta} - u_{h}^{n+\theta}} \dif s
    \\
    -\sum_K \int_{\partial K} \frac{\nu \alpha}{h_K} \envert{\bar{u}_{h}^{n+\theta} - u_{h}^{n+\theta}}^2\dif s
\\
    - \frac{1}{2} \int_{\Gamma_N} \envert{\bar{u}_{h}^{n} \cdot n} \envert{ \bar{u}_{h}^{n+\theta} }^2 \dif s
    - \sum_K \int_K \nu \envert{\nabla u_{h}^{n+\theta}}^2 \dif x.
  \end{multline}
  As in \cref{prop:globenerStab}, there exists an $\alpha > 0$,
  independent of $h_K$, such that the right hand side of
  \cref{eq:afterPresIntbypartsDiscrete_simp_tf} is non-positive. The
  result follows. \qed
\end{proof}

\section{Numerical examples}
\label{sec:compres}

We now demonstrate the performance of the method for a selection of
numerical examples, paying close attention to mass and momentum
conservation, and energy stability.

For all stationary examples considered, exact solutions are known.
For the stationary examples we use a fixed-point iteration with
stopping criterion $|e_p^{i+1} - e_p^{i}|/(e_p^{i+1} + e_p^{i}) \le
{\rm TOL}$, where $e_p^i$ is the pressure error in the $L^2$ norm at
the $i$th iterate, and ${\rm TOL}$ is a given tolerance that we set
to~$10^{-4}$.  All unsteady examples use~$\theta = 1$.  In all
examples we set the penalty parameter to be~$\alpha = 6k^2$.

In the implementation we apply cell-wise static condensation such that
only the degrees-of-freedom associated with the facet spaces appear in
the global system. Compared to standard discontinuous Galerkin
methods, this significantly reduces the size of the global system.  We
could eliminate the facet pressure field and use a BDM element,
see~\citep{Lehrenfeld:2016}, and the BDM normal velocity in place
of~$\bar{u}_{h} \cdot n$. However, we feel that handling all fields in
a hybridized framework offers some simplicity.

Examples have been implemented using the NGSolve finite element
library~\citep{Schoberl:2014}.  All examples use unstructured
simplicial meshes.

\subsection{Kovasznay flow}
\label{ss:tc_kovasznay}

We consider the steady, two-dimensional analytical solution of the
Navier--Stokes equations from Kovasznay~\citep{Kovasznay:1948} on a
domain $\Omega = \del{-0.5, 1} \times \del{-0.5, 1.5}$.  For a
Reynolds number $\mathit{Re}$, let the viscosity be given by $\nu = 1
/ \mathit{Re}$. The solution to the Kovasznay problem is:
\begin{subequations}
  \label{eq:kovasznay}
  \begin{align}
    u_x &= 1 - e^{\lambda x_{1}} \cos(2 \pi x_{2}), \quad
    \\
    u_y &= \frac{\lambda}{2 \pi} e^{\lambda x_{1}} \sin(2\pi x_{2}), \quad
    \\
    p   &= \frac{1}{2} \del{1 - e^{2\lambda x_{1}}} + C,
  \end{align}
\end{subequations}
where $C$ is an arbitrary constant, and where
\begin{equation}
  \lambda
  = \frac{\mathit{Re}}{2} - \del{\frac{\mathit{Re}^2}{4} + 4 \pi^2}^{1/2}
\end{equation}
We choose $C$ such that the mean pressure on $\Omega$ is zero. The
Kovasznay flow solution in \cref{eq:kovasznay} is used to set
Dirichlet boundary conditions for the velocity on~$\partial \Omega$.

The $L^2$-error and rates of convergence are presented in
\cref{tab:kovasznay} for $\mathit{Re} = 40$ using a series of refined
meshes. Optimal rates of convergence are observed for both the
velocity field (order $k + 1$) and pressure field (order $k$). The
divergence of the approximate velocity field is of machine precision
in all cases.

\begin{table}
  \caption{Computed velocity, pressure and divergence errors in the
    $L^2$ norm for the HDG method applied to the Kovasznay
    problem.}
  \label{tab:kovasznay}
  \centering {
    \begin{tabular}{c|ccccc}
      \hline
      & \multicolumn{5}{c}{$k=2$} \\
      Cells
      & $\norm{u_h - u}$ & rate & $\norm{p_h - p}$ & rate & $\norm{\nabla \cdot u_h}$ \\
      \hline
      64   & 1.8e-2 & -   & 1.6e-2 & -   & 3.8e-14 \\
      256  & 2.2e-3 & 3.0 & 4.0e-3 & 2.0 & 6.7e-14 \\
      1024 & 2.8e-4 & 3.0 & 9.8e-4 & 2.0 & 1.3e-13 \\
      4096 & 3.5e-5 & 3.0 & 2.4e-4 & 2.0 & 2.5e-13 \\
      \hline
      \multicolumn{6}{c}{} \\
      \hline
      & \multicolumn{5}{c}{$k=3$} \\
      Cells
      & $\norm{u_h - u}$ & rate & $\norm{p_h - p}$ & rate & $\norm{\nabla \cdot u_h}$ \\
      \hline
      64   & 1.4e-3 & -   & 2.0e-3 & -   & 1.9e-13 \\
      256  & 9.4e-5 & 3.9 & 2.0e-4 & 3.3 & 6.1e-13 \\
      1024 & 5.8e-6 & 4.0 & 2.3e-5 & 3.1 & 7.8e-13 \\
      4096 & 3.6e-7 & 4.0 & 2.8e-6 & 3.1 & 1.6e-12 \\
      \hline
    \end{tabular}
  }
\end{table}

\subsection{Position-dependent Coriolis force}
\label{ss:tc_coriolis}

We now consider the test case from~\cite[Section~3.2]{Linke:2016a}. In
particular, we consider on the unit square $(0, 1) \times (0, 1)$ the
steady Navier--Stokes equations augmented with a position-dependent
Coriolis force: $\nabla \cdot \sigma + 2C \times u = 0$ and $\nabla
\cdot u = 0$, where we set $2 C \times u = - 2 x_{2} (-u_2, u_1)$. On
boundaries we set $u = (1, 0)$. The exact solution to this problem is
given by $p = x_{2}^2 - 1/3$ and $u = (1, 0)$.

It was shown in \citep{Linke:2016a} that the Scott--Vogelius finite
element, in which the velocity is approximated in divergence-free
function spaces, is able to produce the exact velocity field while the
velocity computed using a Taylor--Hood finite element method is
polluted by the pressure error, in part due to the approximate
velocity field not being exactly divergence-free. Furthermore, it is
shown in~\citep{Linke:2016a} that as $\nu \to 0$, the velocity error
increases for the Taylor--Hood finite element method.

In \cref{tab:linke_jcp_2} we show the results obtained using the HDG
method presented in \cref{sec:hdg} for~$k = 2$. It shows the computed
error in the $L^2$ norm for the velocity, pressure and divergence
errors.  Errors in the velocity and velocity divergence are of machine
precision, regardless of~$\nu$. The HDG method therefore obtains the
same quality of solution as produced using the Scott--Vogelius finite
element in~\citep{Linke:2016a}.  We do not consider the $k = 3$ case
because for this discretization the pressure is approximated by
quadratic polynomials and so the pressure error is also of machine
precision.

\begin{table}
  \caption{Computed errors in the $L^2$ norm for the HDG method with
    the position-dependent Coriolis forcing term with different
    viscosity values. Note that the pressure error does not depend on
    the viscosity.}
  \label{tab:linke_jcp_2}
  \centering {\small
    \begin{tabular}{c|cccc|cccc}
      \hline
      & \multicolumn{4}{|c}{$\nu=0.001$, $k=2$}
      & \multicolumn{4}{|c}{$\nu=1$, $k=2$} \\
      Cells
      & $\norm{u_h - u}$ & $\norm{\nabla \cdot u_h}$ & $\norm{p_h - p}$ & rate
                                                     & $\norm{u_h - u}$ & $\norm{\nabla \cdot u_h}$ & $\norm{p_h - p}$ & rate
      \\
      \hline
      64   & 3.7e-15 & 4.6e-14 & 9.0e-4 & -   & 1.3e-14 & 4.6e-14 & 9.0e-4 & -   \\
      256  & 5.1e-15 & 9.3e-14 & 2.3e-4 & 2.0 & 9.9e-15 & 9.2e-14 & 2.3e-4 & 2.0 \\
      1024 & 6.4e-15 & 1.8e-13 & 5.6e-5 & 2.0 & 7.8e-14 & 1.9e-13 & 5.6e-5 & 2.0 \\
      4096 & 2.2e-14 & 6.3e-13 & 1.4e-5 & 2.0 & 2.3e-13 & 3.7e-13 & 1.4e-5 & 2.0 \\
      \hline
    \end{tabular}
  }
\end{table}

\subsection{Pressure-robustness}
\label{ss:pressurerobust}

We next demonstrate that our method is pressure-robust and compare the
results with those obtained using the method
of~\citep{Labeur:2012}. For this we use a test case proposed
in~\citep[Section 6.1]{Lederer:2017}. On the unit square
$(0,1)\times(0,1)$ we consider the steady Navier--Stokes equations
where the boundary conditions and source terms are such that the exact
solution is given by $u = \mathrm{curl} \zeta$, with $\zeta = x_{1}^2
(x_{1} - 1)^2 x_{2}^2 (x_{2} - 1)^2$ and $p = x_{1}^7 + x_{2}^7 -
1/4$. We choose $k = 3$ and vary the viscosity~$\nu$.

For a mixed velocity-pressure approximation, it can be proven that the
method of~\citep{Labeur:2012} results in an approximate velocity field
that is pointwise divergence free, but not $H({\rm div})$-conforming.
As such, the method of~\citep{Labeur:2012} cannot be shown to be
pressure-robust. This is confirmed by the results presented in
\cref{tab:pressure_robust}. For the method of~\citep{Labeur:2012} it
is observed that the error in the velocity field depends on
$\nu^{-1}\norm{p_h - p}$, while the proposed method is pressure
robust. The velocity error does not change with viscosity in
\cref{tab:pressure_robust}.

\begin{table}
  \caption{Computed errors in the $L^2$ norm for the HDG method with
    different viscosity values. A comparison of the proposed method
    with the method of~\citep{Labeur:2012}. Note that the errors in
    the velocity for the proposed method do not depend on $\nu$, in
    contrast with the method of~\citep{Labeur:2012}.}
  \label{tab:pressure_robust}
  \centering {
    \begin{tabular}{c|ccccccc}
      \hline
      & \multicolumn{7}{c}{Proposed Method: $\nu = 0.001$ } \\[1ex]
      Cells
      & $\norm{u_h - u}$ & rate & $\norm{\nabla(u_h - u)}$ & rate & $\norm{p_h - p}$ & rate & $\norm{\nabla \cdot u_h}$ \\
      \hline
      128  & 4.2e-6  &     & 2.8e-4 &     & 4.5e-4 &     & 1.5e-14 \\
      512  & 2.4e-7  & 4.1 & 3.4e-5 & 3.0 & 5.7e-5 & 3.0 & 4.3e-14 \\
      2048 & 1.4e-8  & 4.1 & 4.2e-6 & 3.0 & 7.2e-6 & 3.0 & 2.3e-14 \\
      8192 & 8.5e-10 & 4.1 & 5.2e-7 & 3.0 & 9.0e-7 & 3.0 & 2.4e-14 \\
      \hline
      \multicolumn{8}{c}{} \\
      \hline
      & \multicolumn{7}{c}{Proposed Method: $\nu=1$} \\[1ex]
      Cells
      & $\norm{u_h - u}$ & rate & $\norm{\nabla(u_h - u)}$ & rate & $\norm{p_h - p}$ & rate & $\norm{\nabla \cdot u_h}$ \\
      \hline
      128  & 4.2e-6  & -   & 2.8e-4 & -   & 6.5e-4 & -   & 2.2e-15 \\
      512  & 2.4e-7  & 4.1 & 3.4e-5 & 3.0 & 7.6e-5 & 3.1 & 4.5e-15 \\
      2048 & 1.4e-8  & 4.1 & 4.2e-6 & 3.0 & 9.1e-6 & 3.1 & 8.8e-15 \\
      8192 & 8.5e-10 & 4.1 & 5.2e-7 & 3.0 & 1.1e-6 & 3.0 & 1.8e-14 \\
      \hline
      \multicolumn{8}{c}{} \\[1em]
      \hline
      & \multicolumn{7}{c}{Method of~\citep{Labeur:2012}: $\nu=0.001$} \\[1ex]
      Cells
      & $\norm{u_h - u}$ & rate & $\norm{\nabla(u_h - u)}$ & rate & $\norm{p_h - p}$ & rate & $\norm{\nabla \cdot u_h}$ \\
      \hline
      128  & 6.4e-4 & -   & 5.3e-2 & -   & 4.6e-4 & -   & 2.9e-14 \\
      512  & 4.1e-5 & 4.0 & 6.7e-3 & 3.0 & 5.8e-5 & 3.0 & 4.8e-14 \\
      2048 & 2.6e-6 & 4.0 & 8.4e-4 & 3.0 & 7.2e-6 & 3.0 & 1.7e-14 \\
      8192 & 1.6e-7 & 4.0 & 1.1e-4 & 3.0 & 9.0e-7 & 3.0 & 2.3e-14 \\
      \hline
      \multicolumn{8}{c}{} \\[1ex]
      \hline
      & \multicolumn{7}{c}{Method of~\citep{Labeur:2012}: $\nu=1$} \\[1ex]
      Cells
      & $\norm{u_h - u}$ & rate & $\norm{\nabla(u_h - u)}$ & rate & $\norm{p_h - p}$ & rate & $\norm{\nabla \cdot u_h}$ \\
      \hline
      128  & 3.9e-6  & -   & 2.8e-4 & -   & 6.2e-4 & -   & 2.2e-15 \\
      512  & 2.3e-7  & 4.1 & 3.4e-5 & 3.0 & 7.3e-5 & 3.1 & 4.4e-15 \\
      2048 & 1.4e-8  & 4.1 & 4.2e-6 & 3.0 & 8.8e-6 & 3.1 & 8.9e-15 \\
      8192 & 8.3e-10 & 4.0 & 5.2e-7 & 3.0 & 1.1e-6 & 3.0 & 1.8e-14 \\
      \hline
    \end{tabular}
  }
\end{table}

\subsection{Transient higher-order potential flow}
\label{ss:tc_harmonic}

In this test, taken from \cite[Section~6.6]{Linke:2016b}, we solve the
time dependent Navier--Stokes equations \cref{eq:navsto} on the domain
$\Omega = [-1, 1]^2$. This test case studies the time-dependent exact
velocity $u(t) = \min(t, 1) \nabla \chi$ where $\chi$ is a smooth
harmonic potential given by $\chi = x_{1}^3 x_{2} - x_{2}^3
x_{1}$. The pressure gradient then satisfies $\nabla p = - \nabla
\envert{u}^2/2 - \partial_t \del{\min(t, 1) \nabla \chi}$.  We impose
the exact velocity solution as Dirichlet boundary condition on all
of~$\partial \Omega$.

For the simulations we used a grid with 2048 cells, set the time step
equal to~$\Delta t = 0.01$ and compute the solution on the time
interval $[0, 2]$. \Cref{fig:harmonic} shows the velocity and pressure
errors as a function of time. We used both $k = 2$ and $k = 3$, and
consider $\nu = 1/500$ and $\nu = 1/2000$. We observe that the error
in pressure and velocity is more or less the same regardless of~$\nu$.

Over the computational time interval, using $k = 2$ or $k = 3$ on a
mesh with 2048 cells, for either $\nu = 1/500$ and $\nu = 1/2000$, the
$L^2$-norm of the divergence reaches $1.4 \times 10^{-10}$ in one
point but is otherwise always of the order $10^{-11}$. The momentum
balance, in absolute value, never exceeds~$3.4 \times 10^{-12}$.

\begin{figure}
  \centering
  \subfloat[Velocity error.]{\includegraphics[width=0.49\textwidth]{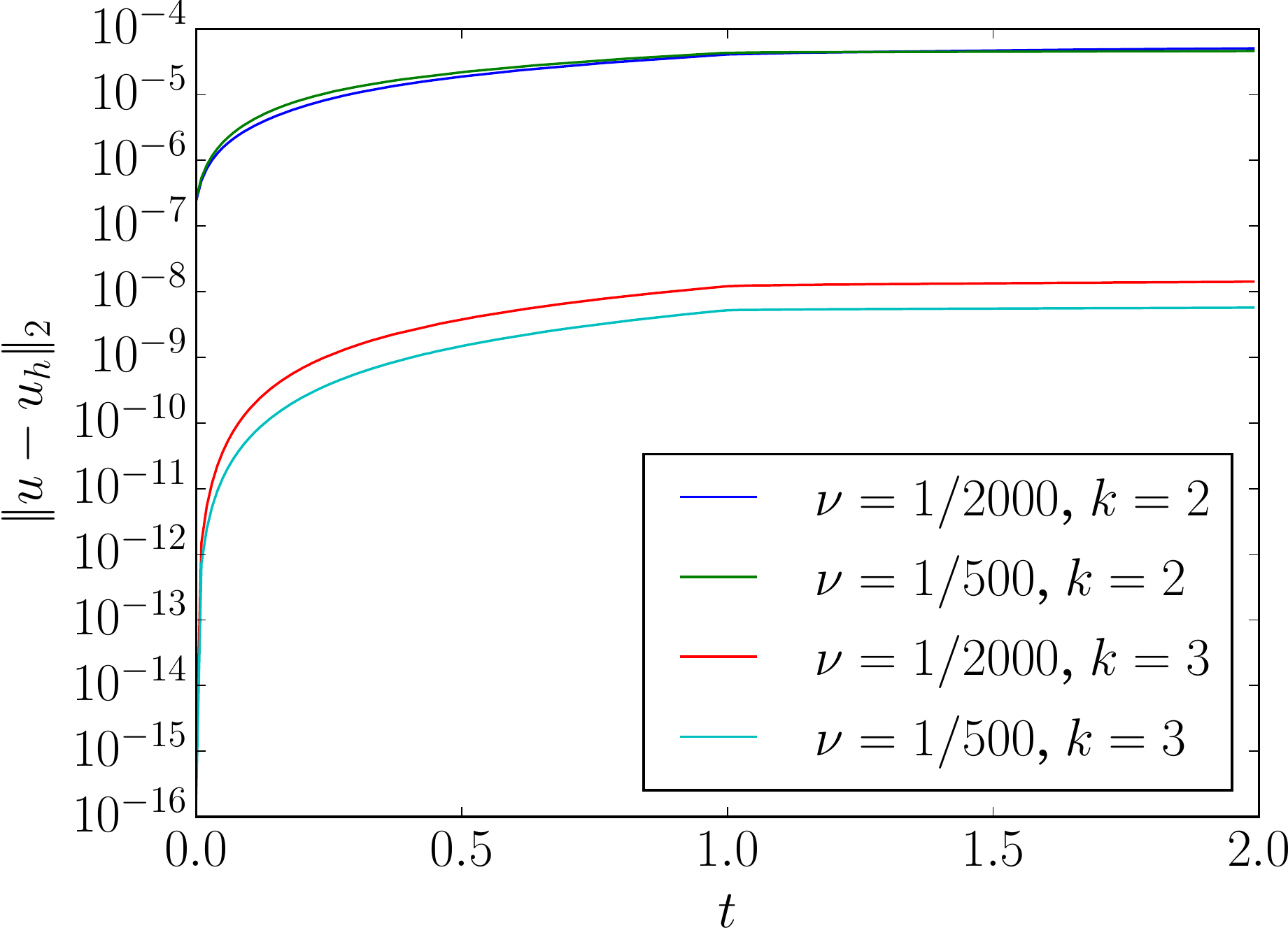}}
  \subfloat[Pressure error.]{\includegraphics[width=0.49\textwidth]{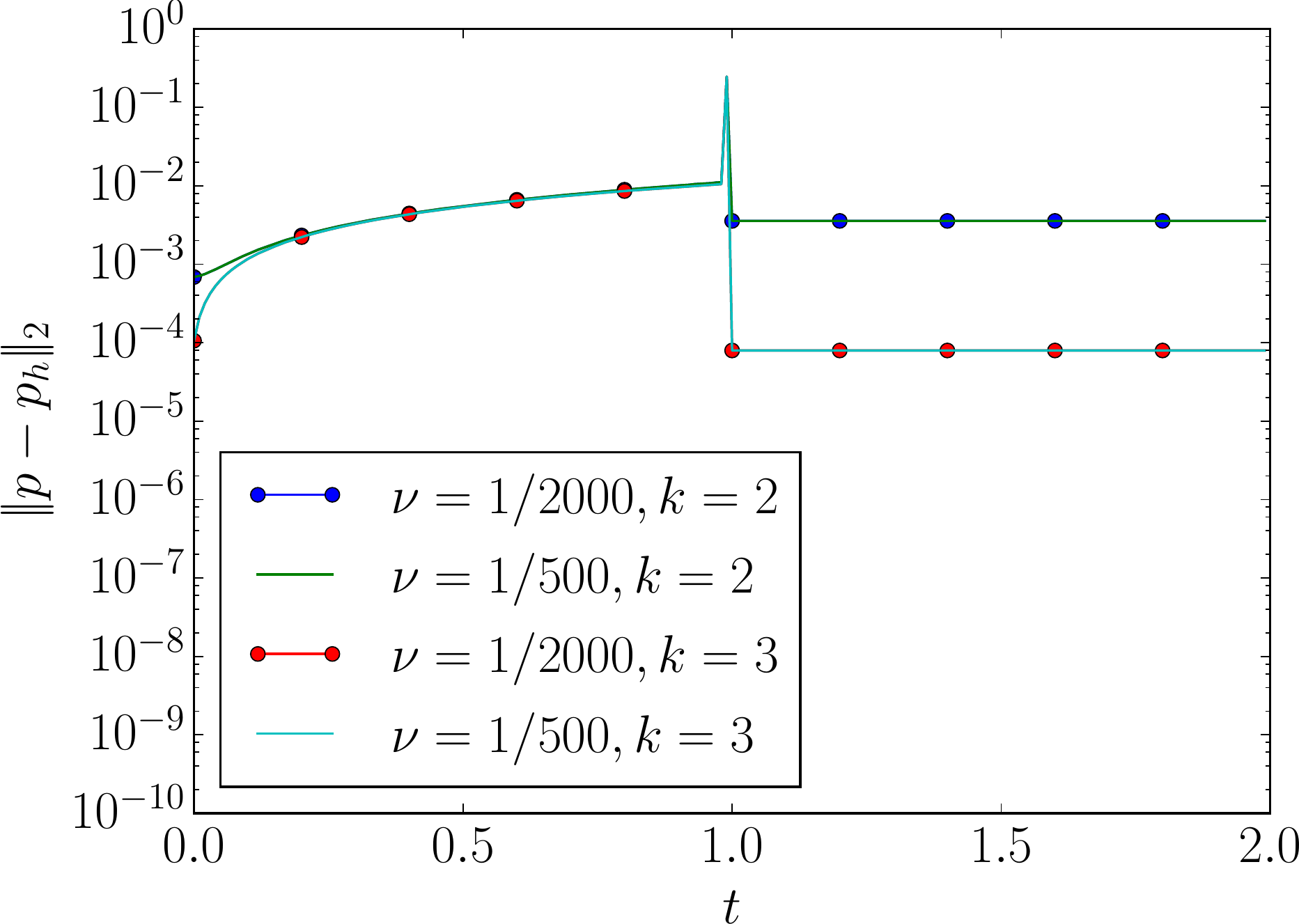}}
  \caption{Velocity and pressure errors in the $L^2$ norm for the
    transient higher-order potential flow test case. Approximations
    were obtained using $k = 2$ and $k = 3$ on a mesh with 2048
    cells.}
  \label{fig:harmonic}
\end{figure}

\subsection{Two-dimensional flow past a circular obstacle}
\label{ss:tc_two_dim_flow}

In this test case we consider flow past a circular obstacle (see
e.g.~\citep{Lehrenfeld:2016,Schaefer:1996}). The domain is a
rectangular channel, $[0, 2.2] \times [0, 0.41]$, with a circular
obstacle of radius $r = 0.05$ centered at~$(0.2, 0.2)$. On the inflow
boundary ($x_{1} = 0$) we prescribe the $x_{1}$-component of the
velocity to be $u_1 = 6 x_{2} (0.41 - x_{2})/0.41^2$. The
$x_{2}$-component of the velocity is prescribed as~$u_2 = 0$.
Homogeneous Dirichlet boundary conditions are applied on the walls
($x_{2} = 0$ and $x_{2} = 0.41$), and on the obstacle. On the outflow
boundary ($x_{1} = 2.2$) we prescribe $\sigma_d \cdot n = 0$. The
viscosity is set as $\nu = 10^{-3}$.  We choose $k = 3$ and set
$\Delta t = 5 \times 10^{-5}$ so that the spatial discretization error
dominates the temporal discretization error. For the initial
condition, we impose the steady Stokes solution of this problem. The
mesh of the domain has 6784 cells and we consider the time
interval~$[0, 5]$.

At each time step we compute the drag and lift coefficients, which are
defined as
\begin{equation}
  \label{eq:dragLift}
  C_D = -\frac{1}{r}\int_{\Gamma_c} \del{\sigma_d \cdot n} \cdot e_1 \dif s,
  \qquad
  C_L = -\frac{1}{r}\int_{\Gamma_c} \del{\sigma_d \cdot n} \cdot e_2 \dif s,
\end{equation}
where $e_{1}$ and $e_{2}$ are unit vectors in the $x_{1}$ and $x_{2}$
directions, respectively, and $\Gamma_C$ is the surface of the
circular object. We compute a maximum drag coefficient of $C_D =
3.23232$ and minimum drag coefficient of $C_D = 3.16583$.  The maximum
and minimum lift coefficients we compute are, respectively, $C_L =
0.98251$ and $C_L = -1.02246$. These are comparable to those found in
literature~\citep{Lehrenfeld:2016,Schaefer:1996}.  The velocity
magnitude at $t = 5$ is shown \cref{fig:velocitymagnitude2D}.
\begin{figure}
  \centering
  \includegraphics[width=0.98\textwidth]{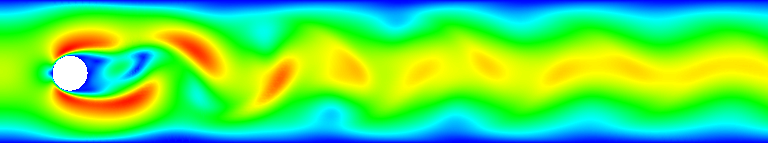}
  \caption{Two-dimensional flow past a cylinder test case: velocity
    magnitude past a two-dimensional circular object in a channel at
    $t = 5$. Approximations were obtained using $k = 3$ on a mesh with
    6784 cells.}
  \label{fig:velocitymagnitude2D}
\end{figure}

\subsection{Three-dimensional flow past a cylinder}
\label{ss:tc_three_dim_flow}

In this test case we consider three-dimensional flow past a cylinder
(see e.g.~\citep{Lehrenfeld:2016,Schaefer:1996}) with a time dependent
inflow velocity. The domain is a cuboid shaped channel $[0, 2.5]
\times [0, 0.41] \times [0, 0.41]$ with a cylinder of radius $r_{\rm
  cyl} = 0.05$ around the $x_{3}$-axis centered at $(x_{1}, x_{2}) =
(0.5, 0.2)$. On the inflow boundary ($x_{1} = 0$) we prescribe the
$x_{1}$-component of the velocity to be $u_1 = 36 \sin(\pi t/ 8) x_{2}
x_{3}(0.41 - x_{2})(0.41 - x_{3})/0.41^4$.  The $x_{2}$- and
$x_{3}$-components of the velocity are prescribed as $u_2 = 0$
and~$u_{3} = 0$. We impose homogeneous Dirichlet boundary conditions
on the walls ($x_{2} = 0$, $x_{2} = 0.41$, $x_{3} = 0$ and $x_{3} =
0.41$) and on the cylinder. On the outflow boundary ($x_{1} = 2.5$) we
prescribe $\sigma_d \cdot n = 0$.  The viscosity is set as~$\nu =
10^{-3}$.

We choose $k = 3$ and set $\Delta t = 5 \times 10^{-4}$ so that the
spatial discretization error dominates the temporal discretization
error. The initial condition is the Stokes solution to this
problem. The mesh has 4091 cells and we compute on the time
interval~$[0, 8]$. At each time step we compute the drag and lift
coefficients, defined by \cref{eq:dragLift}, where $r = 0.41 r_{\rm
  cyl}$ and $\Gamma_C$ is the surface of the cylinder. We compute
maximum drag and lift coefficients of $C_D = 2.98815$ and $C_L =
0.00348$, respectively. Compared to \citet{Schaefer:1996}, in which
the maximum drag and lift coefficients lie in the intervals $C_D \in
[3.2000, 3.3000]$ and $C_L \in [0.0020, 0.0040]$, we slightly
under-predict the drag coefficient, but the lift coefficient lies
within the same interval. \Cref{fig:velocitymagnitude3D} shows the
velocity magnitude at~$t = 4$.
\begin{figure}
  \centering
  \includegraphics[width=0.9\textwidth]{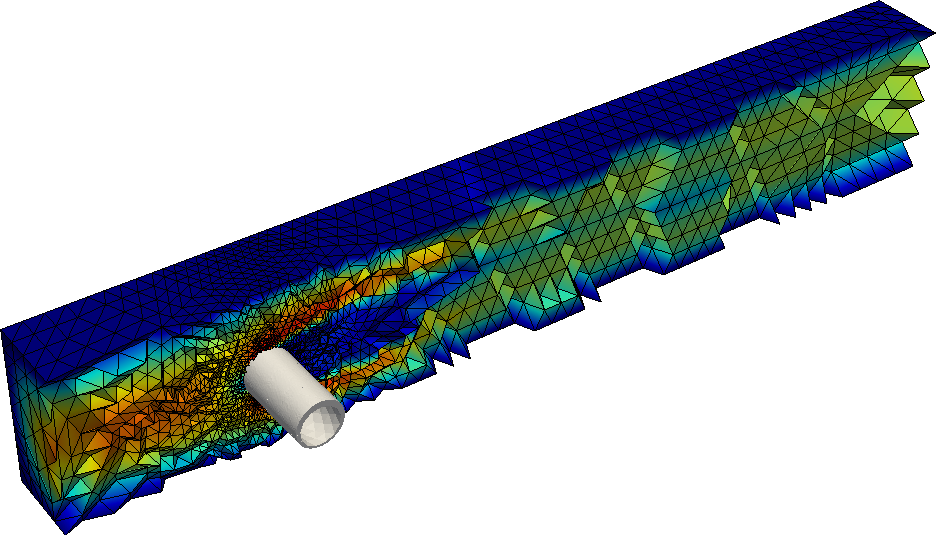}
  \caption{Three-dimensional flow past a cylinder test: slice through
    a 3D channel showing the 3D velocity magnitude past a cylinder in
    a channel at~$t = 4$. Approximations were obtained using $k = 3$
    on a mesh with 4091 cells.}
  \label{fig:velocitymagnitude3D}
\end{figure}

\section{Conclusions}
\label{sec:conclusions}

We have introduced a formulation of a hybridizable discontinuous
Galerkin method for the incompressible Navier--Stokes equations that
computes velocity fields that are pointwise divergence-free. The
construction of solenoidal velocity fields does not require
post-processing or the use of finite dimensional spaces of
divergence-free functions. The pointwise satisfaction of the
continuity equation and the continuity of the normal component of the
velocity field across cell facets allows us to prove that the method
conserves momentum locally (cell-wise) and is energy stable.  This is
in contrast with the closely related method in \citet{Labeur:2012}
which when satisfying the continuity equation pointwise can satisfy
local momentum conservation or global energy stability, but not both
simultaneously. The analysis that we present is supported by a range
of numerical examples in two and three dimensions.

\bibliographystyle{spbasic}
\bibliography{references}
\end{document}